\documentclass{amsart}
\usepackage{setspace}
\usepackage{a4}
\usepackage{amssymb,amsmath,amsthm,latexsym}
\usepackage{amsfonts}
\usepackage{amsfonts}
\usepackage{graphicx}
\usepackage{textcomp}
\usepackage{cite}
\usepackage{enumerate}
\usepackage[mathscr]{euscript}
\usepackage{mathtools}
\newtheorem{theorem}{Theorem}[section]

\newtheorem{conjecture}[theorem]{Conjecture}
\newtheorem{corollary}[theorem] {Corollary}
\newtheorem{definition}[theorem]{Definition}
\newtheorem{example}[theorem]{Example}
\newtheorem{lemma} [theorem]{Lemma}

\newtheorem{proposition}[theorem]{Proposition}

\newtheorem{question}[theorem]{Question}
\title{This is the title}
\usepackage{amssymb}
\usepackage{amssymb}
\usepackage{amssymb}
\usepackage{amssymb}
\usepackage{amsmath}
\usepackage{tikz}
\usepackage{hyperref}
\usepackage{enumerate}
\usepackage{mathtools}
\usepackage{amsmath}
\usepackage{tikz}
\usepackage{amssymb}
\usepackage{amsmath}
\usepackage{tikz}
\usepackage{hyperref}
\usepackage{fancyhdr}
\begin{document}

	\begin{center}
		{\bf{CONTINUOUS WELCH BOUNDS WITH APPLICATIONS}}\\
		K. MAHESH KRISHNA\\
		Department of Humanities and Basic Sciences\\
		Aditya College of Engineering and Technology\\
		 Surampalem, East-Godavari\\
		Andhra Pradesh 533 437 India\\
		Email: kmaheshak@gmail.com \\
		\today
	\end{center}
	
	\hrule
	\vspace{0.5cm}

%--------------------------------------
\textbf{Abstract}:  Let $(\Omega, \mu)$ be a  measure space and $\{\tau_\alpha\}_{\alpha\in \Omega}$ be a 	normalized continuous Bessel  family for a finite dimensional Hilbert  space $\mathcal{H}$ of dimension $d$. If the diagonal $\Delta\coloneqq \{(\alpha, \alpha):\alpha \in \Omega\}$ is measurable in the measure space $\Omega\times \Omega$, then we show that 
	\begin{align*}
		\sup _{\alpha, \beta \in \Omega, \alpha\neq \beta}|\langle \tau_\alpha, \tau_\beta\rangle |^{2m}\geq \frac{1}{(\mu\times\mu)((\Omega\times\Omega)\setminus\Delta)}\left[\frac{	\mu(\Omega)^2}{{d+m-1 \choose m}}-(\mu\times\mu)(\Delta)\right],  \quad \forall m \in \mathbb{N}.
\end{align*}
This improves 47 years old celebrated result of Welch [\textit{IEEE Transactions on  Information Theory, 1974}].  We introduce the notions of continuous cross correlation and frame potential of Bessel family and give applications of continuous Welch bounds to these concepts. We also introduce the notion of continuous Grassmannian frames. 


%cite geometric grassmannian 
%\tableofcontents
%DO GRASSMANNIAN 
\section{Introduction}
In 1974, L. Welch proved the following milestone result which revolutioned the study of finite set of vectors in finite dimensional  Hilbert spaces.
\begin{theorem}\cite{WELCH}\label{WELCHTHEOREM} (\textbf{Welch bounds})
Let $n\geq d$.	If	$\{\tau_j\}_{j=1}^n$  is any collection of  unit vectors in $\mathbb{C}^d$, then
	\begin{align*}
		\sum_{j=1}^n\sum_{k=1}^n|\langle \tau_j, \tau_k\rangle |^{2m}\geq \frac{n^2}{{d+m-1\choose m}}, \quad \forall m \in \mathbb{N}.
	\end{align*}
	In particular,
	\begin{align*}
		\sum_{j=1}^n\sum_{k=1}^n|\langle \tau_j, \tau_k\rangle |^{2}\geq \frac{n^2}{{d}}.
	\end{align*}
Further, 
	\begin{align}\label{FIRST123}
\text{(\textbf{Higher order Welch bounds})}	\quad		\max _{1\leq j,k \leq n, j\neq k}|\langle \tau_j, \tau_k\rangle |^{2m}\geq \frac{1}{n-1}\left[\frac{n}{{d+m-1\choose m}}-1\right], \quad \forall m \in \mathbb{N}.
	\end{align}
	In particular,
	\begin{align*}
	\text{(\textbf{First order Welch bound})}\quad 	\max _{1\leq j,k \leq n, j\neq k}|\langle \tau_j, \tau_k\rangle |^{2}\geq\frac{n-d}{d(n-1)}.
	\end{align*}
\end{theorem}
A very powerful application of Welch bounds is the lower bound on root-mean-square (RMS) absolute cross relation of unit vectors  $\{\tau_j\}_{j=1}^n$ which is defined as 

\begin{align*}
	I_{\text{RMS}} (\{\tau_j\}_{j=1}^n)\coloneqq \left(\frac{1}{n(n-1)}\sum _{1\leq j,k \leq n, j\neq k}|\langle \tau_j, \tau_k\rangle |^{2}\right)^\frac{1}{2}.
\end{align*}
Theorem \ref{WELCHTHEOREM} says that 
\begin{align*}
	I_{\text{RMS}} (\{\tau_j\}_{j=1}^n) \geq \left(\frac{n-d}{d(n-1)}\right)^\frac{1}{2}.
\end{align*}
Another powerful application of Theorem \ref{WELCHTHEOREM} is the lower bound for frame potential which is introduced by Benedetto and Fickus \cite{BENEDETTOFICKUS} and further studied  in  \cite{CASAZZAFICKUSOTHERS, BODMANNHAASPOTENTIAL}. Let us recall that given a collection of unit vectors $\{\tau_j\}_{j=1}^n$, the frame potential is defined as 
	\begin{align*}
FP(\{\tau_j\}_{j=1}^n)\coloneqq	\sum_{j=1}^n\sum_{k=1}^n|\langle \tau_j, \tau_k\rangle |^{2}.
\end{align*}
Theorem  \ref{WELCHTHEOREM}  directly tells 
\begin{align*}
	FP(\{\tau_j\}_{j=1}^n)\geq \frac{n^2}{{d}}.
\end{align*}
There are several practical applications of Theorem \ref{WELCHTHEOREM} such as correlations \cite{SARWATE},  codebooks \cite{DINGFENG}, numerical search algorithms  \cite{XIA, XIACORRECTION}, quantum measurements 
\cite{SCOTTTIGHT}, coding and communications \cite{TROPPDHILLON, STROHMERHEATH}, code division multiple access (CDMA) systems \cite{CHEBIRA1, CHEBIRA2}, wireless systems \cite{YATES}, compressed sensing \cite{TAN}, `game of Sloanes' \cite{JASPERKINGMIXON}, equiangular tight frames \cite{SUSTIKTROPP},  etc.
 \\
%Proof of Welch bounds appears at various places  \cite{MASSEYMITTELHOLZER} \cite{SUSTIK} \cite{MASSEYON}. It also appears implicitly in the celebrated paper .\\
A decade ago, a continuous version of Theorem \ref{WELCHTHEOREM} appeared in the paper \cite{DATTAHOWARD} which states as follows. 
\begin{theorem}\cite{DATTAHOWARD}\label{DATTATHEOREM}
Let $\mathbb{C}	\mathbb{P}^{n-1}$ be the complex projective space and $\mu$ be a normalized measure on $\mathbb{C}	\mathbb{P}^{n-1}$.  If $\{\tau_\alpha\}_{\alpha \in \mathbb{C}	\mathbb{P}^{n-1}}$ is a continuous frame for a $d$-dimensional subspace $\mathcal{H}$  of a Hilbert space $\mathcal{H}_0$, then 
\begin{align*}
	\int_{\mathbb{C}	\mathbb{P}^{n-1}}\int_{\mathbb{C}	\mathbb{P}^{n-1}}|\langle \tau_\alpha, \tau_\beta\rangle|^{2m}\, d \mu(\alpha)\, d \mu(\beta)\geq \frac{1}{{d+m-1\choose m}}, \quad \forall m \in \mathbb{N}.
\end{align*}
\end{theorem}
Drawback of Theorem \ref{DATTATHEOREM} is that it works only for the measures defined on complex projective spaces.  Further, we need a generalization of Inequality (\ref{FIRST123})  for measure spaces.
Therefore it is desirable to improve Theorem \ref{DATTATHEOREM}.and to get a continuous version of Inequality (\ref{FIRST123}) by replacing maximum by supremum. For the sake of completeness, we note that there are some further refinements of Theorem \ref{WELCHTHEOREM}, see \cite{CHRISTENSENDATTAKIM, DATTAWELCHLMA, WALDRONSH}.\\
The goal  of this article is to derive Theorem \ref{WELCHTHEOREM}  for arbitrary measure spaces (Theorem  \ref{CONTINUOUSWELCHMAIN}). We  give some applications of Theorem  \ref{CONTINUOUSWELCHMAIN}. We also ask some  problems for further research.

%define potential and connect with trace
\section{Continuous Welch bounds}
Our proof of the result stated in the abstract is using the theory of continuous frames. This is generalization of frames indexed by discrete sets to measurable sets. Continuous frames are introduced independently by Ali, Antoine and Gazeau \cite{ALIANTOINEGAZEAU} and Kaiser \cite{KAISER}. In the paper,   $\mathbb{K}$ denotes $\mathbb{C}$ or $\mathbb{R}$ and $\mathcal{H}$ denotes a finite dimensional Hilbert space.
\begin{definition}\cite{ALIANTOINEGAZEAU, KAISER}
Let 	$(\Omega, \mu)$ be a measure space. A collection   $\{\tau_\alpha\}_{\alpha\in \Omega}$ in 	a  Hilbert  space $\mathcal{H}$ is said to be a \textbf{continuous frame} (or generalized frame) for $\mathcal{H}$ if the following holds.
\begin{enumerate}[\upshape(i)]
	\item For each $h \in \mathcal{H}$, the map $\Omega \ni \alpha \mapsto \langle h, \tau_\alpha \rangle \in \mathbb{K}$ is measurable.
	\item There are $a,b>0$ such that 
	\begin{align*}
		a\|h\|^2\leq \int_{\Omega}|\langle h, \tau_\alpha \rangle|^2\,d\mu(\alpha)\leq b \|h\|^2, \quad \forall h \in \mathcal{H}.
	\end{align*}
\end{enumerate}
If $a=b$, then the frame is called as a tight frame and if $\|\tau_\alpha\|=1$, $\forall \alpha \in \Omega$, then we say that the frame is normalized. If $a=b=1$, then the frame is called as a Parseval frame. If we do not demand the first inequality in (ii), then we say it is a \textbf{continuous Bessel family} for $\mathcal{H}$. 
\end{definition}
We first observe that there is an abundance of continuous frames for finite dimensional Hilbert spaces. Further, it is known that given any finite mesure space $(\Omega, \mu)$ and a finite dimensional space $\mathcal{H}$, there exists a continuous frame $\{\tau_\alpha\}_{\alpha\in \Omega}$ for $\mathcal{H}$ \cite{RAHIMIDARABYDARVISHI}. Given a continuous Bessel family, the analysis operator
\begin{align*}
	\theta_\tau:\mathcal{H} \ni h \mapsto \theta_\tau h \in \mathcal{L}^2(\Omega);  \quad \theta_\tau h:\Omega \ni \alpha \mapsto  \langle h, \tau_\alpha \rangle \in \mathbb{K}
\end{align*}
is a well-defined bounded linear operator. Its adjoint, the synthesis operator is given by 
\begin{align*}
	\theta_\tau^*:\mathcal{L}^2(\Omega)\ni f \mapsto \int_{\Omega}f (\alpha)\tau_\alpha \,d\mu(\alpha)\in \mathcal{H}.
\end{align*}
By combining analysis and synthesis operators, we get the frame operator, defined as 
\begin{align*}
	S_\tau\coloneqq 	\theta_\tau^*	\theta_\tau:\mathcal{H} \ni h \mapsto\int_{\Omega}\langle h, \tau_\alpha \rangle \tau_\alpha \,d\mu(\alpha)\in \mathcal{H}.
\end{align*}

Note that the integrals are weak integrals (Pettis integrals \cite{TALAGRAND}). Following result captures the trace of frame operator using Bessel family.
%It is widely known in the frame theory community that one can get trace of operator (even for trace class operators in infinite dimensions) not only using orthonormal bases but also using frames (for instace see ). Following is the continuous version of this result.
\begin{theorem}\label{TRACEFORMULA}
Let $\{\tau_\alpha\}_{\alpha\in \Omega}$ be a 	continuous Bessel  family for    $\mathcal{H}$. Then 
	\begin{align*}
	&	\text{Tra}(S_\tau)=\int_\Omega \| \tau_\alpha\|^2\, d \mu(\alpha),\\
		&	\text{Tra}(S_\tau^2)=\int_\Omega\int_\Omega |\langle \tau_\alpha, \tau_\beta\rangle|^2\, d \mu(\alpha)\, d \mu(\beta).
	\end{align*}
\end{theorem}
\begin{proof}
	Let $\{\omega_j\}_{j=1}^d$ be an orthonormal basis for $\mathcal{H}$, where $d$ is the dimension of $\mathcal{H}$. Then
	\begin{align*}
		\text{Tra}(S_\tau)&=\sum_{j=1}^d\langle S_\tau\omega_j, \omega_j \rangle=\sum_{j=1}^d\left \langle \int_\Omega \langle \omega_j, \tau_\alpha \rangle \tau_\alpha\, d \mu(\alpha), \omega_j\right  \rangle \\
		&=\sum_{j=1}^d\int_\Omega \langle \omega_j, \tau_\alpha \rangle\langle \tau_\alpha,\omega_j \rangle\, d \mu(\alpha)=\int_\Omega\left\langle \sum_{j=1}^d\langle \tau_\alpha,\omega_j \rangle\omega_j, \tau_\alpha \right\rangle \, d \mu(\alpha)\\
		&=\int_\Omega \| \tau_\alpha\|^2\, d \mu(\alpha).
	\end{align*}
Further, 

\begin{align*}
	\text{Tra}(S_\tau^2)&=\sum_{j=1}^d\langle S_\tau^2\omega_j, \omega_j \rangle=\sum_{j=1}^d\langle S_\tau\omega_j,  S_\tau\omega_j \rangle=\sum_{j=1}^d\left \langle \int_\Omega \langle \omega_j, \tau_\alpha \rangle \tau_\alpha\, d \mu(\alpha), S_\tau\omega_j\right  \rangle \\
	&=\sum_{j=1}^d\int_\Omega \langle \omega_j, \tau_\alpha \rangle\langle \tau_\alpha,S_\tau\omega_j \rangle\, d \mu(\alpha)=\int_\Omega\left\langle \sum_{j=1}^d\langle \tau_\alpha,S_\tau\omega_j \rangle\omega_j, \tau_\alpha \right\rangle \, d \mu(\alpha)\\
	&=\int_\Omega\left\langle \sum_{j=1}^d\langle S_\tau^*\tau_\alpha,\omega_j \rangle\omega_j, \tau_\alpha \right\rangle \, d \mu(\alpha)=\int_\Omega \langle S_\tau^*\tau_\alpha, \tau_\alpha\rangle \, d \mu(\alpha)\\
	&=\int_\Omega \langle S_\tau\tau_\alpha, \tau_\alpha\rangle \, d \mu(\alpha)=\int_\Omega\left\langle \int_\Omega\left\langle \tau_\alpha,\tau_\beta\right\rangle \tau_\beta\, d \mu(\beta),\tau_\alpha\right\rangle \, d \mu(\alpha)\\
	&=\int_\Omega\int_\Omega |\langle \tau_\alpha, \tau_\beta\rangle|^2\, d \mu(\alpha)\, d \mu(\beta).
\end{align*}
\end{proof}
Note that a finite spanning set is a frame for finite dimensional Hilbert space \cite{HANKORNELSONLARSON}. Thus it is not required to assume any condition on set of vectors in the discrete case to derive Theorem \ref{TRACEFORMULA}. However, we need to assume the Besselness for continuous family of vectors to assure the existence of frame operator. With Theorem \ref{TRACEFORMULA} we derive continuous Welch bounds. First we need a lemma.
\begin{lemma}\label{LEMMAFINITE}
If $\{\tau_\alpha\}_{\alpha\in \Omega}$ is a normalized	continuous Bessel  family for    $\mathcal{H}$ with bound $b$, then $\mu(\Omega)\leq b \operatorname{dim}(\mathcal{H})$. In particular, $\mu(\Omega)<\infty$.
\end{lemma}
\begin{proof}
	Let   $\operatorname{dim}(\mathcal{H})=d$  and $\{\omega_j\}_{j=1}^{d}$ be an orthonormal basis for $\mathcal{H}$. Then 
	\begin{align*}
	\mu(\Omega)&=	\int_\Omega \| \tau_\alpha\|^2\, d \mu(\alpha)=	\int_\Omega \sum_{j=1}^{d}|\langle x_\alpha, \omega_j \rangle|^2\, d \mu(\alpha)=\sum_{j=1}^{d}\int_\Omega|\langle x_\alpha, \omega_j \rangle|^2\, d \mu(\alpha)\\
	&=\sum_{j=1}^{d}\int_\Omega|\langle \omega_j, x_\alpha \rangle|^2\, d \mu(\alpha)\leq \sum_{j=1}^{d}b\|\omega_j\|^2=bd.
	\end{align*}
\end{proof}
\begin{theorem}\label{FIRSTORDERCONTINUOUS}
	 Let $(\Omega, \mu)$ be a  measure space and $\{\tau_\alpha\}_{\alpha\in \Omega}$ be a 	normalized continuous Bessel  family for  $\mathcal{H}$ of dimension $d$. If the diagonal $\Delta\coloneqq \{(\alpha, \alpha):\alpha \in \Omega\}$ is measurable in the measure space $\Omega\times \Omega$, then 
	\begin{align}\label{1}
\int_{\Omega\times\Omega}|\langle \tau_\alpha, \tau_\beta\rangle|^2\, d(\mu\times\mu)(\alpha,\beta)=\int_{\Omega}\int_{\Omega}|\langle \tau_\alpha, \tau_\beta\rangle|^{2}\, d \mu(\alpha)\, d \mu(\beta)\geq \frac{\mu(\Omega)^2}{d}.
\end{align}	
Equality holds in Inequality (\ref{1}) if and only if $\{\tau_\alpha\}_{\alpha\in \Omega}$ is a tight continuous frame.
Further, we have the \textbf{first order continuous Welch bound}
	\begin{align*}
		\sup _{\alpha, \beta \in \Omega, \alpha\neq \beta}|\langle \tau_\alpha, \tau_\beta\rangle |^{2}\geq \frac{1}{(\mu\times\mu)((\Omega\times\Omega)\setminus\Delta)}\left[\frac{\mu(\Omega)^2}{d}-(\mu\times\mu)(\Delta)\right].
\end{align*}
\end{theorem}
\begin{proof}
Let $\lambda_1, \dots, \lambda_d$ be eigenvalues of the frame operator $	S_{\tau}$. Then $\lambda_1, \dots, \lambda_d\geq0$. Now   using the diagonalizability of $	S_{\tau}$, Cauchy-Schwarz inequality and Theorem  \ref{TRACEFORMULA} we get 	
\begin{align*}
\mu(\Omega)^2&=\left(\int_\Omega \| \tau_\alpha\|^2\, d \mu(\alpha)\right)^2=	(\operatorname{Tra}(S_{\tau}))^2=\left(\sum_{k=1}^d
\lambda_k\right)^2\leq d \sum_{k=1}^d
\lambda_k^2\\
&=d\operatorname{Tra}(S^2_{\tau})=d\int_\Omega\int_\Omega |\langle \tau_\alpha, \tau_\beta\rangle|^2\, d \mu(\alpha)\, d \mu(\beta).	
\end{align*}
Equality holds if and only if we have equality  in Cauchy-Schwarz inequality if and only if the frame is tight. Since the measure is finite (Lemma \ref{LEMMAFINITE}), using Fubini's theorem,
\begin{align*}
\frac{	\mu(\Omega)^2}{d}&=\int_\Omega\int_\Omega |\langle \tau_\alpha, \tau_\beta\rangle|^2\, d \mu(\alpha)\, d \mu(\beta)=\int_{\Omega\times\Omega}|\langle \tau_\alpha, \tau_\beta\rangle|^2\, d(\mu\times\mu)(\alpha,\beta)\\
	&=\int_{\Delta}|\langle \tau_\alpha, \tau_\beta\rangle|^2\, d(\mu\times\mu)(\alpha,\beta)+\int_{(\Omega\times\Omega)\setminus\Delta}|\langle \tau_\alpha, \tau_\beta\rangle|^2\, d(\mu\times\mu)(\alpha,\beta)\\
	&=\int_{\Delta}|\langle \tau_\alpha, \tau_\alpha\rangle|^2\, d(\mu\times\mu)(\alpha,\beta)+\int_{(\Omega\times\Omega)\setminus\Delta}|\langle \tau_\alpha, \tau_\beta\rangle|^2\, d(\mu\times\mu)(\alpha,\beta)\\
	&=(\mu\times\mu)(\Delta)+\int_{(\Omega\times\Omega)\setminus\Delta}|\langle \tau_\alpha, \tau_\beta\rangle|^2\, d(\mu\times\mu)(\alpha,\beta)\\
	&\leq(\mu\times\mu)(\Delta)+\sup _{\alpha, \beta \in \Omega, \alpha\neq \beta}|\langle \tau_\alpha, \tau_\beta\rangle |^{2}(\mu\times\mu)((\Omega\times\Omega)\setminus\Delta).	
\end{align*}
which gives the required inequality after rearrangement. 
\end{proof}
Under the stronger assumption that $\{\tau_\alpha\}_{\alpha\in \Omega}$   is a continuous frame for $\mathcal{H}$, Inequality (\ref{1})  appears in Chapter 16 of \cite{WALDONBOOK}. We now illustrate Theorem  \ref{FIRSTORDERCONTINUOUS}  using the following example.
\begin{example}
	Let $\Omega\coloneqq [0,2\pi]$ and $\mu$ be the Lebesgue measure on $\Omega$. Define 
	\begin{align*}
		\tau_\alpha \coloneqq (\cos \alpha, \sin \alpha ), \quad \forall \alpha \in \Omega.
	\end{align*}
	Then 
	\begin{align*}
		\int_\Omega |\langle (x,y), \tau_\alpha \rangle |^2\, d\alpha&=\int_{0} ^{2\pi} |\langle  (x,y), (\cos \alpha, \sin \alpha ) \rangle |^2\, d\alpha=\int_{0} ^{2\pi} (x\cos \alpha+ y \sin \alpha )^2\, d\alpha\\
		&=\int_{0} ^{2\pi} (x^2\cos^2 \alpha+ y^2 \sin^2 \alpha +2 xy \sin \alpha \cos \alpha )\, d\alpha\\
		&=\pi (x^2+y^2)=	\pi \left\| (x,y) \right\|^2, \quad \forall (x,y) \in \mathbb{R}^2.
	\end{align*}
	Therefore $\{\tau_\alpha\}_{\alpha \in \Omega}$ is a normalized continuous frame for $\mathbb{R}^2$ \cite{WALDONBOOK}. Next we verify inequalities in Theorem \ref{FIRSTORDERCONTINUOUS}: 
	\begin{align*}
		&	\int_{\Omega}\int_{\Omega}|\langle \tau_\alpha, \tau_\beta\rangle|^{2}\, d \mu(\alpha)\, d \mu(\beta)=\int_{0} ^{2\pi}\int_{0} ^{2\pi}|\langle  (\cos \alpha, \sin \alpha ), (\cos \beta, \sin \beta) \rangle |^2\, d\alpha\, d\beta\\
		&=\int_{0} ^{2\pi}\int_{0} ^{2\pi} (\cos \alpha \cos \beta+\sin \alpha \sin \beta)^2\, d\alpha\, d\beta\\
		&=\left(\int_{0} ^{2\pi}\cos^2 \alpha\, d\alpha\right)\left(\int_{0} ^{2\pi}\cos^2 \beta\, d\beta \right)+2 \left(\int_{0} ^{2\pi}\cos \alpha \sin  \alpha \, d\alpha \right)\left(\int_{0} ^{2\pi}\cos\beta \sin  \beta\, d\beta \right)+\\
		&\quad \left(\int_{0} ^{2\pi}\sin^2 \alpha\, d\alpha\right)\left(\int_{0} ^{2\pi}\sin^2 \beta\, d\beta \right)\\
		&=2\pi^2=\frac{(2\pi)^2}{2}=\frac{\mu(\Omega)^2}{d}
	\end{align*}
	and 
	
	\begin{align*}
		\sup _{\alpha, \beta \in \Omega, \alpha\neq \beta}|\langle \tau_\alpha, \tau_\beta\rangle |^{2}&=	\sup _{\alpha, \beta \in [0,2\pi], \alpha\neq \beta}|\langle  (\cos \alpha, \sin \alpha ), (\cos \beta, \sin \beta) \rangle |^2\\
		&=	\sup _{\alpha, \beta \in [0,2\pi], \alpha\neq \beta}|\cos \alpha \cos \beta+\sin \alpha \sin \beta|^2\\
		&=\sup _{\alpha, \beta \in [0,2\pi], \alpha\neq \beta}|\cos^2(\alpha-\beta)|=1>\frac{1}{4\pi^2}\left[\frac{4\pi^2}{2}-0\right]\\
		&=\frac{1}{(\mu\times\mu)((\Omega\times\Omega)\setminus\Delta)}\left[\frac{\mu(\Omega)^2}{d}-(\mu\times\mu)(\Delta)\right].
	\end{align*}
\end{example}
Our next goal is to derive higher order continuous Welch bounds. We are going to use the  following result. 
\begin{theorem}\cite{COMON, BOCCI}\label{SYMMETRICTENSORDIMENSION}
	If $\mathcal{V}$ is a vector space of dimension $d$ and $\text{Sym}^m(\mathcal{V})$ denotes the vector space of symmetric m-tensors, then 
\begin{align*}
	\text{dim}(\text{Sym}^m(\mathcal{V}))={d+m-1 \choose m}, \quad \forall m \in \mathbb{N}.
\end{align*}
\end{theorem}
\begin{theorem}\label{CONTINUOUSWELCHMAIN}
		 Let $(\Omega, \mu)$ be a  measure space and $\{\tau_\alpha\}_{\alpha\in \Omega}$ be a 	normalized continuous Bessel  family for $\mathcal{H}$ of dimension $d$. If the diagonal $\Delta\coloneqq \{(\alpha, \alpha):\alpha \in \Omega\}$ is measurable in the measure space $\Omega\times \Omega$, then 
\begin{align}\label{WELCHCONTINUOUS1}
\int_{\Omega\times\Omega}|\langle \tau_\alpha, \tau_\beta\rangle|^{2m}\, d(\mu\times\mu)(\alpha,\beta)=	\int_{\Omega}\int_{\Omega}|\langle \tau_\alpha, \tau_\beta\rangle|^{2m}\, d \mu(\alpha)\, d \mu(\beta)\geq \frac{\mu(\Omega)^2}{{d+m-1\choose m}}, \quad \forall m \in \mathbb{N}.
\end{align}	
Equality holds  in Inequality  (\ref{WELCHCONTINUOUS1}) if and only if $\{\tau_\alpha\}_{\alpha\in \Omega}$ is a tight continuous frame.
Further, we have the \textbf{higher order continuous Welch bounds} 
\begin{align}\label{WELCHCONTINUOUS2}
	\sup _{\alpha, \beta \in \Omega, \alpha\neq \beta}|\langle \tau_\alpha, \tau_\beta\rangle |^{2m}\geq \frac{1}{(\mu\times\mu)((\Omega\times\Omega)\setminus\Delta)}	\left[\frac{	\mu(\Omega)^2}{{d+m-1 \choose m}}-(\mu\times\mu)(\Delta)\right],  \quad \forall m \in \mathbb{N}.
\end{align}
\end{theorem}
\begin{proof}
First note that $\{\tau_\alpha\}_{\alpha\in \Omega}$ is a normalized continuous Bessel  family for the Hilbert  space $\text{Sym}^m(\mathcal{H})$. We execute the proof of Theorem \ref{FIRSTORDERCONTINUOUS}  for the space $\text{Sym}^m(\mathcal{H})$. Let $\lambda_1, \dots, \lambda_{\text{dim}(\text{Sym}^m(\mathcal{H}))}$ be eigenvalues of $S_{\tau}$.	Then using Theorem  \ref{SYMMETRICTENSORDIMENSION}  we get 
\begin{align*}
	\mu(\Omega)^2&=\left(\int_\Omega \| \tau_\alpha\|^{2m}\, d \mu(\alpha)\right)^2=\left(\int_\Omega \| \tau_\alpha^{\otimes m}\|^2\, d \mu(\alpha)\right)^2=	(\operatorname{Tra}(S_{\tau}))^2\\
	&=\left(\sum_{k=1}^{\text{dim}(\text{Sym}^m(\mathcal{H}))}
	\lambda_k\right)^2\leq \text{dim}(\text{Sym}^m(\mathcal{H})) \sum_{k=1}^{\text{dim}(\text{Sym}^m(\mathcal{H}))}
	\lambda_k^2\\
	&={d+m-1 \choose m}\operatorname{Tra}(S^2_{\tau})={d+m-1 \choose m}\int_\Omega\int_\Omega |\langle \tau_\alpha^{\otimes m}, \tau_\beta^{\otimes m}\rangle|^2\, d \mu(\alpha)\, d \mu(\beta)\\
	&={d+m-1 \choose m}\int_\Omega\int_\Omega |\langle \tau_\alpha, \tau_\beta\rangle|^{2m}\, d \mu(\alpha)\, d \mu(\beta)	
\end{align*}
and hence

\begin{align*}
	\frac{	\mu(\Omega)^2}{{d+m-1 \choose m}}&=\int_\Omega\int_\Omega |\langle \tau_\alpha, \tau_\beta\rangle|^{2m}\, d \mu(\alpha)\, d \mu(\beta)=\int_{\Omega\times\Omega}|\langle \tau_\alpha, \tau_\beta\rangle|^{2m}\, d(\mu\times\mu)(\alpha,\beta)\\
	&=\int_{\Delta}|\langle \tau_\alpha, \tau_\alpha\rangle|^{2m}\, d(\mu\times\mu)(\alpha,\beta)+\int_{(\Omega\times\Omega)\setminus\Delta}|\langle \tau_\alpha, \tau_\beta\rangle|^{2m}\, d(\mu\times\mu)(\alpha,\beta)\\
	&=(\mu\times\mu)(\Delta)+\int_{(\Omega\times\Omega)\setminus\Delta}|\langle \tau_\alpha, \tau_\beta\rangle|^{2m}\, d(\mu\times\mu)(\alpha,\beta)\\
	&\leq(\mu\times\mu)(\Delta)+\sup _{\alpha, \beta \in \Omega, \alpha\neq \beta}|\langle \tau_\alpha, \tau_\beta\rangle |^{2m}(\mu\times\mu)((\Omega\times\Omega)\setminus\Delta)	
\end{align*}
which gives Inequality (\ref{WELCHCONTINUOUS2}).
\end{proof}

\begin{corollary}
	Theorem  \ref{WELCHTHEOREM}    is a corollary of Theorem 	\ref{CONTINUOUSWELCHMAIN}.
\end{corollary}
\begin{proof}
	Take $\Omega=\{1,\dots,n\} $ and $\mu$ as the counting measure.
\end{proof}
\begin{corollary}
	Theorem   \ref{DATTATHEOREM}     is a corollary of Theorem 		\ref{CONTINUOUSWELCHMAIN}.
\end{corollary}
\begin{proof}
	Take $\Omega=\mathbb{C}	\mathbb{P}^{n-1}$ and $\mu$ as the normalized  measure on $\mathbb{C}	\mathbb{P}^{n-1}$.
\end{proof}
We observe that given a measure space $\Omega$, the diagonal $\Delta$ need not be measurable (see \cite{DRAVECKY}). This is the reason behind the measurability of diagonal in Theorem 	\ref{CONTINUOUSWELCHMAIN}. Further, we see that the measurability of the diagonal  $\Delta$ was used only in deriving Inequality (\ref{WELCHCONTINUOUS2})  and not in Inequality (\ref{WELCHCONTINUOUS1}).  \\
In \cite{WALDRON2003}, Waldron generalized Welch bounds to vectors which need not be normalized. In the following result we state such a result for continuous Bessel family whose proof is similar to the proof of Theorem 	\ref{CONTINUOUSWELCHMAIN}.
\begin{theorem}\label{CONTINUOUSWELCHMAINSECOND}
	 Let $(\Omega, \mu)$ be a $\sigma$-finite measure space and $\{\tau_\alpha\}_{\alpha\in \Omega}$ be a 	 continuous Bessel  family for $\mathcal{H}$ of dimension $d$. If the diagonal $\Delta\coloneqq \{(\alpha, \alpha):\alpha \in \Omega\}$ is measurable in the measure space $\Omega\times \Omega$, then 
\begin{align}\label{WELCHCONTINUOUS3}
\int_{\Omega\times\Omega}|\langle \tau_\alpha, \tau_\beta\rangle|^{2m}\, d(\mu\times\mu)(\alpha,\beta)=	\int_{\Omega}\int_{\Omega}|\langle \tau_\alpha, \tau_\beta\rangle|^{2m}\, d \mu(\alpha)\, d \mu(\beta)\geq \frac{1}{{d+m-1\choose m}}\left(\int_\Omega \| \tau_\alpha\|^{2m}\, d \mu(\alpha)\right)^2, \quad \forall m \in \mathbb{N}.
\end{align}	
Equality in Inequality (\ref{WELCHCONTINUOUS3})  holds if and only if $\{\tau_\alpha\}_{\alpha\in \Omega}$ is a tight continuous frame.
Further, we have the \textbf{generalized higher order continuous Welch bounds} 
\begin{align}\label{WELCHCONTINUOUS4}
	\sup _{\alpha, \beta \in \Omega, \alpha\neq \beta}|\langle \tau_\alpha, \tau_\beta\rangle |^{2m}\geq \frac{1}{(\mu\times\mu)((\Omega\times\Omega)\setminus\Delta)}	\left[\frac{1}{{d+m-1 \choose m}}\left(\int_\Omega \| \tau_\alpha\|^{2m}\, d \mu(\alpha)\right)^2-\int_{\Delta}\| \tau_\alpha\|^{4m}\, d(\mu\times\mu)(\alpha,\beta)\right], 
\end{align}	
for all $m \in \mathbb{N}.$
\end{theorem}
Note that we imposed $\sigma$-finiteness of measure in Theorem  	\ref{CONTINUOUSWELCHMAINSECOND} to use Fubini's theorem whereas we derived in Lemma \ref{LEMMAFINITE} that measure is finite for normalized continuous Bessel family. Also note that  Theorem  	\ref{CONTINUOUSWELCHMAINSECOND} remains valid as long as Fubini's theorem is valid (for instance, it is valid for complete measure spaces). Since Fubini's theorem is not valid for arbitrary measure spaces, we are finally left with the following problem.	
\begin{question}
	\textbf{Classify measure spaces $(\Omega, \mu) $ such that Theorem \ref{CONTINUOUSWELCHMAINSECOND}  holds?} In other words, \textbf{given a measure space $(\Omega, \mu) $, does the validity of Inequality (\ref{WELCHCONTINUOUS3}) or Inequality (\ref{WELCHCONTINUOUS4})  implies  conditions on meausre space $(\Omega, \mu) $, say $\sigma$-finite?}
\end{question}
In a recent work, Christensen, Datta and Kim derived Welch bounds for dual frames \cite{CHRISTENSENDATTAKIM}. We now extend this result to continuous frames. For this we recall the notion of dual frame. A continuous frame $\{\omega_\alpha\}_{\alpha\in \Omega}$  for $\mathcal{H}$ is said to be a dual for a continuous frame $\{\tau_\alpha\}_{\alpha\in \Omega}$  for $\mathcal{H}$ if $\theta_\omega^*\theta_\tau=I_\mathcal{H}$ or $\theta_\tau^*\theta_\omega=I_\mathcal{H}$, the identity operator on $\mathcal{H}$. In terms of weak integrals, this is same as 
\begin{align*}
\int_{\Omega}\langle h, \tau _\alpha \rangle \omega_\alpha \, d \mu (\alpha)=h, ~ \forall h \in \mathcal{H}	\quad \text{ or } \quad \int_{\Omega}\langle h, \omega _\alpha \rangle \tau_\alpha \, d \mu (\alpha)=h, ~ \forall h \in \mathcal{H}.
\end{align*}
We now see that the frame $\{S_\tau^{-1}\tau_\alpha\}_{\alpha\in \Omega}$ is always a dual to a frame  $\{\tau_\alpha\}_{\alpha\in \Omega}$ for $\mathcal{H}$. Further, if $\{\omega_\alpha\}_{\alpha\in \Omega}$ is any  dual for  $\{\tau_\alpha\}_{\alpha\in \Omega}$, then 
\begin{align}\label{BEST}
	\int_{\Omega}|\langle h, \omega_\alpha\rangle |^2\,d\mu(\alpha)\geq 	\int_{\Omega}|\langle h,  S_\tau^{-1}\tau_\alpha\rangle|^2\,d\mu(\alpha), \quad \forall h \in \mathcal{H}.
\end{align}
We need two more results before we derive continuous Welch bounds for dual frames.
\begin{theorem}\label{TRACETHEOREM}
If $\{\tau_\alpha\}_{\alpha\in \Omega}$ is a  continuous frame for  $\mathcal{H}$, then for any linear operator 	$T:\mathcal{H}\to \mathcal{H}$, we have
\begin{align*}
	\operatorname{Tra}(T)=\int_{\Omega}\langle  TS_\tau^{-\frac{1}{2}}\tau _\alpha,  S_\tau^{-\frac{1}{2}}\tau _\alpha\rangle  \, d \mu (\alpha).
\end{align*}
\end{theorem}
\begin{proof}
	First we prove the theorem for Parseval frames. Assume that $\{\tau_\alpha\}_{\alpha\in \Omega}$ is Parseval. 	Let $\{\omega_j\}_{j=1}^d$ be an orthonormal basis for $\mathcal{H}$, where $d$ is the dimension of $\mathcal{H}$. Then 
	\begin{align*}
	\operatorname{Tra}(T)&=\sum_{j=1}^d\langle T\omega_j, \omega_j \rangle=\sum_{j=1}^d\left \langle \int_\Omega \langle T\omega_j, \tau_\alpha \rangle \tau_\alpha\, d \mu(\alpha), \omega_j\right  \rangle \\
	&=\sum_{j=1}^d\int_\Omega \langle T\omega_j, \tau_\alpha \rangle\langle \tau_\alpha,\omega_j \rangle\, d \mu(\alpha)=\int_\Omega\left\langle \sum_{j=1}^d\langle \tau_\alpha,\omega_j \rangle T\omega_j, \tau_\alpha \right\rangle \, d \mu(\alpha)\\
	&=\int_\Omega \langle T\tau_\alpha, \tau_\alpha \rangle \, d \mu(\alpha).		
	\end{align*}
Now the theorem follows by noting that $\{S_\tau^{-1/2}\tau_\alpha\}_{\alpha\in \Omega}$ is a Parseval frame for $\mathcal{H}$.
\end{proof}
\begin{theorem}\label{DIMBOUNDED}
If $\{\omega_\alpha\}_{\alpha\in \Omega}$   is a dual continuous frame for $\{\tau_\alpha\}_{\alpha\in \Omega}$, then 
\begin{align*}
	\int_{\Omega}\int_{\Omega}|\langle \tau_\alpha, \omega_\beta\rangle|^{2}\, d \mu(\alpha)\, d \mu(\beta)\geq \operatorname{dim}(\mathcal{H}).
\end{align*}	
\end{theorem}
\begin{proof}
Inequality 	(\ref{BEST}) says that 
\begin{align*}
		\int_{\Omega}|\langle \tau_\alpha, \omega_\beta \rangle|^2\,d\mu(\beta)\geq 	\int_{\Omega}|\langle \tau_\alpha,  S_\tau^{-1}\tau_\beta\rangle|^2\,d\mu(\beta), \quad \forall \alpha \in \Omega.
\end{align*}
Therefore 
\begin{align*}
	\int_{\Omega}\int_{\Omega}|\langle \tau_\alpha, \omega_\beta\rangle|^{2}\, d \mu(\beta)\, d \mu(\alpha)\geq	\int_{\Omega}\int_{\Omega}|\langle \tau_\alpha,  S_\tau^{-1}\tau_\beta\rangle|^2\,d\mu(\beta)\, d \mu(\alpha).
\end{align*}
Now we simplify the right side and  use Theorem \ref{TRACETHEOREM} to get 
\begin{align*}
\int_{\Omega}\int_{\Omega}|\langle \tau_\alpha,  S_\tau^{-1}\tau_\beta\rangle|^2\,d\mu(\beta)\, d \mu(\alpha)&=\int_{\Omega}\int_{\Omega}|\langle S_\tau^{-\frac{1}{2}}\tau_\alpha,  S_\tau^{-\frac{1}{2}}\tau_\beta\rangle|^2\,d\mu(\beta)\, d \mu(\alpha)	\\
&=\int_{\Omega}\|S_\tau^{-\frac{1}{2}}\tau_\alpha\|^2\, d \mu(\alpha)=\int_{\Omega}\langle S_\tau^{-\frac{1}{2}}\tau_\alpha, S_\tau^{-\frac{1}{2}}\tau_\alpha\rangle\, d \mu(\alpha)\\
&=\operatorname{Tra}(I_\mathcal{H})=\operatorname{dim}(\mathcal{H}).
\end{align*}
\end{proof}
\begin{theorem}\label{DUALWELCHCONTINUOUS}
 Let  $\{\tau_\alpha\}_{\alpha\in \Omega}$ be a 	 continuous frame   for  $\mathcal{H}$ of dimension $d$. Assume that $\{\omega_\alpha\}_{\alpha\in \Omega}$   is a dual continuous frame for $\{\tau_\alpha\}_{\alpha\in \Omega}$ and 
 \begin{align*}
 	\langle \tau_\alpha, \omega_\alpha \rangle =	\langle \tau_\beta, \omega_\beta \rangle, \quad \forall \alpha, \beta \in \Omega.
 \end{align*}
If the diagonal $\Delta\coloneqq \{(\alpha, \alpha):\alpha \in \Omega\}$ is measurable in the measure space $\Omega\times \Omega$, then 	
\begin{align*}
	\sup _{\alpha, \beta \in \Omega, \alpha\neq \beta}|\langle \tau_\alpha, \omega_\beta\rangle |^{2}\geq \frac{d(\mu(\Omega)^2-d(\mu \times \mu)(\Delta)}{\mu(\Omega)^2(\mu\times\mu)((\Omega\times\Omega)\setminus\Delta)}.
\end{align*}
\end{theorem}
\begin{proof}
	Since  $\{\omega_\alpha\}_{\alpha\in \Omega}$  is a dual for $\{\tau_\alpha\}_{\alpha\in \Omega}$ we have $\theta_\omega^*\theta_\tau=I_\mathcal{H}$. 
	Let $\{\rho_j\}_{j=1}^d$ be an orthonormal basis for $\mathcal{H}$. Then
	\begin{align*}
		d&=\operatorname{dim}(\mathcal{H})=\sum_{j=1}^d\langle \rho_j, \rho_j \rangle=\sum_{j=1}^d\left\langle \int_{\Omega}\langle \rho_j, \tau _\alpha \rangle \omega_\alpha \, d \mu (\alpha), \rho_j \right \rangle\\
		&=\int_{\Omega}\sum_{j=1}^{d}\langle \rho_j, \tau _\alpha \rangle \langle \omega_\alpha, \rho_j \rangle \, d \mu (\alpha)=\int_{\Omega}\left\langle \omega_\alpha, \sum_{j=1}^{d}\langle \tau _\alpha , \rho_j\rangle \rho_j\right\rangle \, d \mu (\alpha)\\
		&=\int_{\Omega}\langle \omega_\alpha, \tau_\alpha\rangle \,d\mu(\alpha)=\int_{\Omega}\langle \tau_\alpha, \omega_\alpha\rangle \,d\mu(\alpha).
	\end{align*}
Set $\gamma \coloneqq \langle \tau_\alpha, \omega_\alpha \rangle$ which is independent of $\alpha$ by the assumption. Then 
\begin{align*}
	\int_{\Delta}|\langle \tau_\alpha, \omega_\alpha\rangle |^2\,d(\mu\times\mu)(\alpha,\beta)	&=\int_{\Delta}|\gamma|^2\,d(\mu\times\mu)(\alpha,\beta)=\int_{\Delta}\left|\frac{1}{\mu(\Omega)}\int_{\Omega}\langle \tau_\alpha, \omega_\alpha\rangle\,d\mu(\alpha)\right|^2\,d(\mu\times\mu)(\alpha,\beta)\\
	&=\int_{\Delta}\left|\frac{d}{\mu(\Omega)}\right|^2\,d(\mu\times\mu)(\alpha,\beta)=\frac{d^2(\mu\times \mu)(\Delta)}{\mu(\Omega)^2}.
\end{align*}
Theorem \ref{DIMBOUNDED} then gives 
\begin{align*}
&	\sup _{\alpha, \beta \in \Omega, \alpha\neq \beta}|\langle \tau_\alpha, \omega_\beta\rangle |^{2}\geq \frac{1}{(\mu\times\mu)((\Omega\times\Omega)\setminus\Delta)}	\int_{(\Omega\times\Omega)\setminus\Delta}|\langle \tau_\alpha, \omega_\beta\rangle|^2\, d(\mu\times\mu)(\alpha,\beta)\\
	&~=\frac{1}{(\mu\times\mu)((\Omega\times\Omega)\setminus\Delta)}\left[\int_{\Omega}\int_{\Omega}|\langle \tau_\alpha, \omega_\beta\rangle|^{2}\, d \mu(\alpha)\, d \mu(\beta)-\int_{\Delta}|\langle \tau_\alpha, \omega_\alpha\rangle|^2\,d(\mu\times\mu)(\alpha,\beta)\right]\\
	&~\geq \frac{1}{(\mu\times\mu)((\Omega\times\Omega)\setminus\Delta)}\left[d-\frac{d^2(\mu \times \mu)(\Delta)}{\mu(\Omega)^2}\right]=\frac{d}{(\mu\times\mu)((\Omega\times\Omega)\setminus\Delta)}\left[1-\frac{d(\mu \times \mu)(\Delta)}{\mu(\Omega)^2}\right].
\end{align*}
\end{proof}
\begin{corollary}
Let  $\{\tau_\alpha\}_{\alpha\in \Omega}$ be a 	 continuous frame   for  $\mathcal{H}$ of dimension $d$. Assume that $\{\omega_\alpha\}_{\alpha\in \Omega}$   is a dual continuous frame for $\{\tau_\alpha\}_{\alpha\in \Omega}$. If the diagonal $\Delta\coloneqq \{(\alpha, \alpha):\alpha \in \Omega\}$ is measurable in the measure space $\Omega\times \Omega$, then 
\begin{align*}
	\sup _{\alpha, \beta \in \Omega, \alpha\neq \beta}|\langle \tau_\alpha, \omega_\beta\rangle |^{2}\geq	\frac{1}{(\mu\times\mu)((\Omega\times\Omega)\setminus\Delta)}\left[d-\int_{\Delta}|\langle \tau_\alpha, \omega_\alpha\rangle|^2\,d(\mu\times\mu)(\alpha,\beta)\right].
\end{align*}	
\end{corollary}
Higher order continuous Welch bounds leads  to the following question which we do not have answer at present.
\begin{question}
	Is there  a higher order version of Theorem  \ref{DUALWELCHCONTINUOUS} like Theorem  \ref{CONTINUOUSWELCHMAIN}?
\end{question}
%In the next result we derive Theorem   \ref{CONTINUOUSWELCHMAIN}  in its most general form.
%\begin{theorem}
	
%\end{theorem}
%\begin{proof}
	
%\end{proof}
%\begin{corollary}
	
%\end{corollary}
%\begin{corollary}
	
%\end{corollary}

It is natural to ask whether we have continuous  Welch bounds by replacing natural number $m$ in Theorem  \ref{CONTINUOUSWELCHMAIN} by arbitrary positive real $r$. We now derive such results.
%Following theorem gives such a result but we will not able to remove the frame operator to given vectors. 
In the discrete case, the first result for normalized tight frames is derived in \cite{HAIKINZAMIRGAVISH} and the second result is derived in \cite{EHLEROKOUDJOU}.
\begin{theorem}
Let $\{\tau_\alpha\}_{\alpha\in \Omega}$ be a 	normalized continuous Bessel  family for $\mathcal{H}$ of dimension $d$. Then
\begin{align*}
\frac{1}{\mu(\Omega)}	\operatorname{Tra}(\theta_\tau^*\theta_\tau)^r\geq \left(\frac{\mu(\Omega)}{d}\right)^{d-1} ,\quad \forall r \in [1, \infty)
\end{align*}	
and 
\begin{align*}
	\frac{1}{\mu(\Omega)}	\operatorname{Tra}(\theta_\tau^*\theta_\tau)^r\leq \left(\frac{\mu(\Omega)}{d}\right)^{d-1} ,\quad \forall r \in (0,1).
\end{align*}
\end{theorem}
\begin{proof}
Let $\lambda_1, \dots, \lambda_{\text{dim}(\text{Sym}^m(\mathcal{H}))}$ be eigenvalues of $	S_{\tau}$. Let $r \in [1, \infty)$. Using Jensen's inequality	
\begin{align*}
	\left(\frac{1}{d}\sum_{k=1}^d\lambda_k\right)^r\leq \frac{1}{d}\sum_{k=1}^d\lambda_k^r.
\end{align*}
Since $S_{\tau}$ is diagonalizable we get 
\begin{align*}
\left(\frac{\mu(\Omega)}{d}\right)^r=\left(\frac{1}{d}\int_\Omega \| \tau_\alpha\|^2\, d \mu(\alpha)\right)^r=	\left(\frac{1}{d}\operatorname{Tra}(S_\tau)\right)^r\leq \frac{1}{d}\operatorname{Tra}(S_\tau)^r=\frac{1}{d}\operatorname{Tra}(\theta_\tau^*\theta_\tau)^r.
\end{align*}
Similarly the case $ r \in (0,1)  $ follows by using Jensen's inequality.
\end{proof}
\begin{theorem}\label{PWELCH}
Let $2<p<\infty$. Let $(\Omega, \mu)$ be a  measure space and $\{\tau_\alpha\}_{\alpha\in \Omega}$ be a 	normalized continuous Bessel  family for $\mathcal{H}$ of dimension $d$. If the diagonal $\Delta\coloneqq \{(\alpha, \alpha):\alpha \in \Omega\}$ is measurable in the measure space $\Omega\times \Omega$, then 	
\begin{align*}
\int_{\Omega\times\Omega}|\langle \tau_\alpha, \tau_\beta\rangle|^{p}\, d(\mu\times\mu)(\alpha,\beta)&=	\int_{\Omega}\int_{\Omega}|\langle \tau_\alpha, \tau_\beta\rangle|^{p}\, d \mu(\alpha)\, d \mu(\beta)\\
&\geq  	\frac{1}{(\mu\times \mu)((\Omega\times\Omega)\setminus\Delta)^{\frac{p}{2}-1}}	\left(\frac{\mu(\Omega)^2}{d}-(\mu\times\mu)(\Delta)\right)^\frac{p}{2}+(\mu\times\mu)(\Delta).
\end{align*}
\end{theorem}
\begin{proof}
Define $r\coloneqq 2p/(p-2)$ and $q$ be the conjugate index of $p/2$. Then $q=r/2$. Using Theorem \ref{FIRSTORDERCONTINUOUS} and Holder's inequality, we  have 
	\begin{align*}
\frac{\mu(\Omega)^2}{d}-(\mu\times\mu)(\Delta)&\leq	\int_{(\Omega\times\Omega)\setminus\Delta}|\langle \tau_\alpha, \tau_\beta\rangle|^2\, d(\mu\times\mu)(\alpha,\beta)\\
&\leq \left(\int_{(\Omega\times\Omega)\setminus\Delta}||\langle \tau_\alpha, \tau_\beta\rangle|^2|^\frac{p}{2}\, d(\mu\times\mu)(\alpha,\beta)\right)^\frac{2}{p}\left(\int_{(\Omega\times\Omega)\setminus\Delta}d(\mu\times\mu)(\alpha,\beta)\right)^\frac{1}{q}\\
&=\left(\int_{(\Omega\times\Omega)\setminus\Delta}|\langle \tau_\alpha, \tau_\beta\rangle|^p\, d(\mu\times\mu)(\alpha,\beta)\right)^\frac{2}{p}(\mu\times \mu)((\Omega\times\Omega)\setminus\Delta)^\frac{1}{q}\\	
&=\left(\int_{(\Omega\times\Omega)\setminus\Delta}|\langle \tau_\alpha, \tau_\beta\rangle|^p\, d(\mu\times\mu)(\alpha,\beta)\right)^\frac{2}{p}(\mu\times \mu)((\Omega\times\Omega)\setminus\Delta)^\frac{2}{r}\\	
&=\left(\int_{(\Omega\times\Omega)\setminus\Delta}|\langle \tau_\alpha, \tau_\beta\rangle|^p\, d(\mu\times\mu)(\alpha,\beta)\right)^\frac{2}{p}(\mu\times \mu)((\Omega\times\Omega)\setminus\Delta)^\frac{p-2}{p}	
\end{align*}
which gives 
\begin{align*}
	\left(\frac{\mu(\Omega)^2}{d}-(\mu\times\mu)(\Delta)\right)^\frac{p}{2}\leq \left(\int_{(\Omega\times\Omega)\setminus\Delta}|\langle \tau_\alpha, \tau_\beta\rangle|^p\, d(\mu\times\mu)(\alpha,\beta)\right)(\mu\times \mu)((\Omega\times\Omega)\setminus\Delta)^{\frac{p}{2}-1}.
\end{align*}
Therefore 
\begin{align*}
&\frac{1}{(\mu\times \mu)((\Omega\times\Omega)\setminus\Delta)^{\frac{p}{2}-1}}	\left(\frac{\mu(\Omega)^2}{d}-(\mu\times\mu)(\Delta)\right)^\frac{p}{2}+(\mu\times\mu)(\Delta)\\
	&~=\frac{1}{(\mu\times \mu)((\Omega\times\Omega)\setminus\Delta)^{\frac{p}{2}-1}}	\left(\frac{\mu(\Omega)^2}{d}-(\mu\times\mu)(\Delta)\right)^\frac{p}{2}+\int_{\Delta}|\langle \tau_\alpha, \tau_\alpha\rangle|^p\, d(\mu\times\mu)(\alpha,\beta)\\
	&~\leq \int_{(\Omega\times\Omega)\setminus\Delta}|\langle \tau_\alpha, \tau_\beta\rangle|^p\, d(\mu\times\mu)(\alpha,\beta)+\int_{\Delta}|\langle \tau_\alpha, \tau_\alpha\rangle|^p\, d(\mu\times\mu)(\alpha,\beta)\\
	&~=\int_{\Omega\times\Omega}|\langle \tau_\alpha, \tau_\beta\rangle|^{p}\, d(\mu\times\mu)(\alpha,\beta).
\end{align*}
\end{proof}
%\begin{corollary}
	
%\end{corollary}
%A slight variation in the proof of Theorem \ref{PWELCH}  gives the following result.
%\begin{theorem}
%Let $2<p<\infty$.  Let $(\Omega, \mu)$ be a  measure space and $\{\tau_\alpha\}_{\alpha\in \Omega}$ be a 	 continuous Bessel  family for $\mathcal{H}$ of dimension $d$. If the diagonal $\Delta\coloneqq \{(\alpha, \alpha):\alpha \in \Omega\}$ is measurable in the measure space $\Omega\times \Omega$, then 	
%\end{theorem}
%\begin{proof}
	
%\end{proof}
%\begin{corollary}
	
%\end{corollary}
There are four more bounds which are in line with Welch bounds. To state them we need a definition.
\begin{definition}\cite{JASPERKINGMIXON}
Given $d\in \mathbb{N}$, define \textbf{Gerzon's bound}
\begin{align*}
	\mathcal{Z}(d, \mathbb{K})\coloneqq 
	\left\{ \begin{array}{cc} 
		d^2 & \quad \text{if} \quad \mathbb{K} =\mathbb{C}\\
	\frac{d(d+1)}{2} & \quad \text{if} \quad \mathbb{K} =\mathbb{R}.\\
	\end{array} \right.
\end{align*}	
\end{definition}
\begin{theorem}\cite{JASPERKINGMIXON, XIACORRECTION, MUKKAVILLISABHAWALERKIPAAZHANG, SOLTANALIAN, BUKHCOX, CONWAYHARDINSLOANE, HAASHAMMENMIXON, RANKIN}  \label{LEVENSTEINBOUND}
Define $m\coloneqq \operatorname{dim}_{\mathbb{R}}(\mathbb{K})/2$.	If	$\{\tau_j\}_{j=1}^n$  is any collection of  unit vectors in $\mathbb{K}^d$, then
\begin{enumerate}[\upshape(i)]
	\item (\textbf{Bukh-Cox bound})
	\begin{align*}
		\max _{1\leq j,k \leq n, j\neq k}|\langle \tau_j, \tau_k\rangle |\geq \frac{\mathcal{Z}(n-d, \mathbb{K})}{n(1+m(n-d-1)\sqrt{m^{-1}+n-d})-\mathcal{Z}(n-d, \mathbb{K})}\quad \text{if} \quad n>d.
	\end{align*}
	\item (\textbf{Orthoplex/Rankin bound})	
	\begin{align*}
		\max _{1\leq j,k \leq n, j\neq k}|\langle \tau_j, \tau_k\rangle |\geq\frac{1}{\sqrt{d}} \quad \text{if} \quad n>\mathcal{Z}(d, \mathbb{K}).
	\end{align*}
	\item (\textbf{Levenstein bound})	
	\begin{align*}
		\max _{1\leq j,k \leq n, j\neq k}|\langle \tau_j, \tau_k\rangle |\geq \sqrt{\frac{n(m+1)-d(md+1)}{(n-d)(md+1)}} \quad \text{if} \quad n>\mathcal{Z}(d, \mathbb{K}).
	\end{align*}
\item (\textbf{Exponential bound})
	\begin{align*}
	\max _{1\leq j,k \leq n, j\neq k}|\langle \tau_j, \tau_k\rangle |\geq 1-2n^{\frac{-1}{d-1}}.
\end{align*}
\end{enumerate}	
\end{theorem}
Theorem \ref{LEVENSTEINBOUND}   leads to the following problem.
\begin{question}
	Whether there is a  continuous version of Theorem \ref{LEVENSTEINBOUND}?. In particular, does there exists a continuous version of 
\begin{enumerate}[\upshape(i)]
	\item Bukh-Cox bound?
	\item Orthoplex/Rankin bound?
	\item Levenstein bound?
	\item Exponential bound?
\end{enumerate}		
\end{question}

\section{Applications}
Our first application of Theorem  \ref{CONTINUOUSWELCHMAIN}  is to the continuous version of RMS correlation of vectors which we define as follows.
\begin{definition}\label{CRMSDEFINITION}
Let  $\{\tau_\alpha\}_{\alpha\in \Omega}$ be a 	normalized continuous Bessel  family for  $\mathcal{H}$. If the diagonal $\Delta$ is measurable, then the \textbf{continuous root-mean-square} (CRMS) absolute cross relation of $\{\tau_\alpha\}_{\alpha\in \Omega}$  is defined as 	
\begin{align*}
I_{\text{CRMS}}(\{\tau_\alpha\}_{\alpha\in \Omega})\coloneqq	\left(\frac{1}{(\mu\times\mu)((\Omega\times\Omega)\setminus\Delta)}\int_{(\Omega\times\Omega)\setminus\Delta}|\langle \tau_\alpha, \tau_\beta\rangle|^2\, d(\mu\times\mu)(\alpha,\beta)\right)^\frac{1}{2}.
\end{align*}
\end{definition}
Theorem 	\ref{CONTINUOUSWELCHMAIN} now gives the following estimate.
\begin{proposition}
	Under the set up as in Definition \ref{CRMSDEFINITION}, one has 
	\begin{align*}
1\geq I_{\text{CRMS}}(\{\tau_\alpha\}_{\alpha\in \Omega})		\geq \left(\frac{1}{(\mu\times\mu)((\Omega\times\Omega)\setminus\Delta)}\left[\frac{\mu(\Omega)^2}{d}-(\mu\times\mu)(\Delta)\right]\right)^\frac{1}{2}.
	\end{align*}
\end{proposition}
Our second application  of Theorem  \ref{CONTINUOUSWELCHMAIN} is to the continuous version of frame potential which is  defined as follows.
\begin{definition}\label{CONTINUOUSPOTENTIALDEFINITION}
 Let  $\{\tau_\alpha\}_{\alpha\in \Omega}$ be a 	normalized continuous Bessel  family for  $\mathcal{H}$. The \textbf{continuous frame potential} of  $\{\tau_\alpha\}_{\alpha\in \Omega}$ is defined as 
 \begin{align*}
 	FP(\{\tau_\alpha\}_{\alpha\in \Omega})\coloneqq \int_{\Omega}\int_{\Omega}|\langle \tau_\alpha, \tau_\beta\rangle|^{2}\, d \mu(\alpha)\, d \mu(\beta).
 \end{align*}	
\end{definition}
Note that  the order of integration does not matter in Definition \ref{CONTINUOUSPOTENTIALDEFINITION}. Further, finiteness of measure says that potential is finite. In general, it is difficult to find potential using Definition \ref{CONTINUOUSPOTENTIALDEFINITION}. Following theorem simplifies it to a greater extent.
\begin{theorem}
 If  $\{\tau_\alpha\}_{\alpha\in \Omega}$ is a 	normalized continuous Bessel  family for  $\mathcal{H}$, then 
 \begin{align*}
 	FP(\{\tau_\alpha\}_{\alpha\in \Omega})=	\text{Tra}(S_\tau^2)=	\text{Tra}((\theta_\tau^*\theta_\tau)^2).
 \end{align*}	
\end{theorem}
\begin{proof}
	This follows from Theorem \ref{TRACEFORMULA}.
\end{proof}
Using Theorem 	\ref{CONTINUOUSWELCHMAIN} we have following estimates.
\begin{proposition}\label{FPESTIMATE}
 Given a normalized continuous Bessel  family	$\{\tau_\alpha\}_{\alpha\in \Omega}$ for $\mathcal{H}$, one has 
\begin{align*}
	\frac{\mu(\Omega)^2}{d}\leq FP(\{\tau_\alpha\}_{\alpha\in \Omega})\leq \mu(\Omega)^2.
\end{align*}	
Further, if the diagonal $\Delta$ is measurable, then one also has 
\begin{align*}
(\mu\times\mu)(\Delta)\leq FP(\{\tau_\alpha\}_{\alpha\in \Omega})\leq \mu(\Omega)^2.
\end{align*}
\end{proposition}
%\begin{proof}
	%\end{proof}
	Proposition \ref{FPESTIMATE} and the study of paper \cite{BENEDETTOFICKUS} leads to the following problem.
\begin{question}
\textbf{Is  there  a characterization of continuous frames using continuous frame potential (like Theorem 7.1 in \cite{BENEDETTOFICKUS})?}
\end{question}
Our third application of Theorem  \ref{CONTINUOUSWELCHMAIN} is to the continuous frame correlations defined as follows.
\begin{definition}
Let  $\{\tau_\alpha\}_{\alpha\in \Omega}$ be a 	normalized continuous frame for $\mathcal{H}$.	We define the  \textbf{continuous frame correlation} of $\{\tau_\alpha\}_{\alpha\in \Omega}$ as 
\begin{align*}
	\mathcal{M}(\{\tau_\alpha\}_{\alpha\in \Omega})\coloneqq \sup_{\alpha, \beta \in \Omega, \alpha\neq \beta}|\langle\tau_\alpha, \tau_\beta \rangle|.
\end{align*}
\end{definition}
%Note that Lemma \ref{LEMMAFINITE} says that $\mathcal{M}(\{\tau_\alpha\}_{\alpha\in \Omega})<\infty$. 
In discrete frame theory the notion which comes along with frame correlation is the notion of Grassmannian frames defined in \cite{STROHMERHEATH}. We next set up the notion of continuous Grassmannian frames. 
\begin{definition}
 A normalized continuous frame $\{\tau_\alpha\}_{\alpha\in \Omega}$ for   $\mathcal{H}$ is said to be a \textbf{continuous Grassmannian frame} for  $\mathcal{H}$ if 
\begin{align*}
	\mathcal{M}(\{\tau_\alpha\}_{\alpha\in \Omega})=\inf\left\{\mathcal{M}(\{\omega_\alpha\}_{\alpha\in \Omega}):\{\omega_\alpha\}_{\alpha\in \Omega}\text{ is a normalized continuous frame for }\mathcal{H} \right\}.	
\end{align*}		
\end{definition}
Using compactness and continuity arguments it is known that Grassmannian frames exist in every dimension with any number of vectors (greater than or equal to dimension) \cite{BENEDETTONKOLESAR}. However we can not use this argument for measures. Therefore we state the following open problem.
\begin{question}
	\textbf{Classify measure spaces and (finite dimensional) Hilbert spaces so that continuous Grassmannian frames exist}.
\end{question}
The notion which is associated to Grassmannian frames is the  notion of equiangular frames (see \cite{STROHMERHEATH}). For the continuous case, we set the definition as follows.
\begin{definition}
 A continuous frame $\{\tau_\alpha\}_{\alpha\in \Omega}$ for  $\mathcal{H}$ is said to be 	\textbf{$\gamma$-equiangular} if there exists $\gamma\geq0$ such that
\begin{align*}
|\langle\tau_\alpha, \tau_\beta \rangle|=\gamma, \quad \forall 	\alpha, \beta \in \Omega, \alpha\neq \beta.
\end{align*}
\end{definition}
There is a celebrated Zauner's conjecture for equiangular tight frames (see \cite{APPLEBY}). For the purpose of record, we set the continuous version of Zauner's conjecture as follows.
\begin{conjecture}
	(\textbf{Continuous Zauner's conjecture}) \textbf{For a given measure space  $(\Omega, \mu)$ and for  every $d\in \mathbb{N}$, there exists a $\gamma$-equiangular tight continuous frame $\{\tau_\alpha\}_{\alpha\in \Omega}$  for $\mathbb{C}^d$ such that $\mu(\Omega)=d^2$}.
\end{conjecture}

\begin{theorem}\label{CNG2}
Let  $\{\tau_\alpha\}_{\alpha\in \Omega}$ be a 	normalized continuous frame for  $\mathcal{H}$. Then 
\begin{align}\label{EQUIANGULARINEQUALITY}
		\mathcal{M}(\{\tau_\alpha\}_{\alpha\in \Omega})\geq \left(\frac{1}{(\mu\times\mu)((\Omega\times\Omega)\setminus\Delta)}\left[\frac{\mu(\Omega)^2}{d}-(\mu\times\mu)(\Delta)\right]\right)^\frac{1}{2}\eqqcolon\gamma.
\end{align}
If the frame is $\gamma$-equiangular, then we have equality in Inequality (\ref{EQUIANGULARINEQUALITY}).
\end{theorem}
In the case of (discrete) Grassmannian frames, the converse statement of  Theorem \ref{CNG2} is valid (see \cite{STROHMERHEATH}). There are also relations between number of elements in the frame and dimension of the space (Theorem 2.3 in \cite{STROHMERHEATH}). We do not  know any such \textbf{relation between measure of $\Omega$ and the dimension of $\mathcal{H}$}.

\section*{Acknowledgements}
I thank Dr. P. Sam Johnson, Department of Mathematical and Computational Sciences, National Institute of Technology Karnataka (NITK),  Surathkal for some discussions.

%Reason is the following. The operator $A$ is constructed as follows. Let $\mathcal{X}$ be a Banach space and assume that $\mathcal{X}=\mathcal{Y} \oplus \mathcal{Z}$ for some closed subspaces $\mathcal{Y}$ and $\mathcal{Z}$ of $\mathcal{X}$. Then the projections $P:\mathcal{Y} \oplus \mathcal{Z}\ni y\oplus z \to y \in \mathcal{X}$, $Q:\mathcal{Y} \oplus \mathcal{Z}\ni y\oplus z \to Z \in \mathcal{X}$  are bounded linear operators. Define $A\coloneqq 2P$ which is linear. Then
%\begin{align*}
%	A-I_\mathcal{X}=2P-(P+Q)=P-Q.
%\end{align*}
%Now for $x=y\oplus z\in \mathcal{Y} \oplus \mathcal{Z}$, we have
%\begin{align*}
%	\|Ax-x\|=\|Px-Qx\|=\|y\oplus z\| \quad \text{ and } \quad \|Ax\|=\|2Px\|=2\|y\|.
%\end{align*}
%Thus there is no $\beta\geq 0$ such that 
%\begin{align*}
%		\|Ax-x\|=\|y\oplus z\|\leq 2\beta\|y\|=\beta  \|Ax\|,\quad \forall x=y\oplus z\in \mathcal{Y} \oplus \mathcal{Z}.
%\end{align*}

 \bibliographystyle{plain}
 \bibliography{reference.bib}

\end{document}